\documentclass[11pt]{amsart} 

\usepackage{amsmath,amssymb,amsthm,amsfonts,verbatim}
\usepackage{enumerate}
\usepackage{color}
\usepackage{graphicx}

\usepackage{tikz-cd}

\theoremstyle{plain}
\newtheorem{theorem}{Theorem}[section]

\newtheorem{proposition}[theorem]{Proposition}

\newtheorem{lemma}[theorem]{Lemma}

\newtheorem{remark}[theorem]{Remark}

\newtheorem{corollary}[theorem]{Corollary}
\theoremstyle{definition}
\newtheorem{definition}[theorem]{Definition}

\newcommand{\nc}{\newcommand}
\nc{\dmo}{\DeclareMathOperator}

\nc{\Q}{\mathbb{Q}}
\nc{\R}{\mathbb{R}}
\nc{\RP}{\mathbb{RP}^1}
\nc{\Z}{\mathbb{Z}}
\nc{\ZZ}{\mathbb{Z}}
\nc{\C}{\mathbb{C}}
\nc{\cS}{\mathcal{S}}
\nc{\iso}{\cong}
\dmo{\Mod}{Mod}
\dmo{\Ig}{\mathcal{I}_g}
\dmo{\Span}{span}
\dmo{\Diff}{Diff}
\dmo{\Homeo}{Homeo}
\dmo{\dist}{dist}
\dmo\BDiff{BDiff}
\dmo\SO{SO}
\dmo\slide{sl}
\dmo\im{im}
\dmo\id{id}
\dmo\Fix{Fix}
\dmo\Stab{Stab}
\dmo\Mcg{Mcg}
\dmo{\Hg}{\mathcal{H}_g}
\dmo{\Tg}{\mathcal{T}_g}

\renewcommand{\epsilon}{\varepsilon}
\nc{\coloneq}{\mathrel{\mathop:}\mkern-1.2mu=}
\nc{\margin}[1]{\marginpar{\scriptsize #1}}
\nc{\para}[1]{\bigskip\noindent\textbf{#1}}

\begin{document}
\title{Rigidity and Flexibility for Handlebody Groups}
\author{Sebastian Hensel}
\date{\today}
\begin{abstract}
  We show that finite index subgroups of the handlebody group are rigid
  in their ambient mapping class group: 
  any injective map of a finite index subgroup of the genus
  $g$ handlebody group into the genus $g$ mapping class group is conjugation
  by a mapping class group element. 

  On the other hand, we construct an injection of the genus $g$
  handlebody group into a genus $h>g$ mapping class group which is not
  conjugate into a handlebody group.
\end{abstract}
\maketitle

\section{Introduction}
\label{sec:intro}

Homomorphisms, and in particular injections, between mapping class
groups have received considerable attention over the last years. See
\cite{AS} for a survey, and e.g.  \cite{ALS, AS2, HK, I, IMcC, K,
  McC2} for examples of results. A guiding theme in this subject is to
try and imitate (super)rigidity results from the theory of
lattices in Lie groups.  For example, under suitable complexity
bounds, the only injections between mapping class groups arise from 
``obvious'' topological operations on surfaces.

\smallskip In this article we investigate rigidity phenomena from a
slightly different point of view. Namely, we let the mapping class group
play the role of the ``ambient Lie group'', and show rigidity of
subgroups. To be precise, by rigidity we here mean the following.
\begin{definition}
  Let $\Gamma$ be a subgroup of the mapping class group $\mathrm{Mcg}(\Sigma_g)$ of a closed genus $g$ surface $\Sigma_g$. We say that
  $\Gamma$ is \emph{rigid in $\mathrm{Mcg}(\Sigma_g)$} if every
  injective map $f: \Gamma \to \mathrm{Mcg}(\Sigma_g)$ is (the
  restriction of) an inner automorphism of $\mathrm{Mcg}(\Sigma_g)$.
\end{definition}
We focus on an important, topologically motivated subgroup of $\mathrm{Mcg}(\Sigma_g)$, namely the the \emph{handlebody group}
$\mathcal{H}_g<\mathrm{Mcg}(\Sigma_g)$. It consists of all those mapping
classes which extend to a given handlebody $V$ with boundary
$\Sigma_g$.  We show.
\begin{theorem}[Rigidity]\label{thm:rigidity-handlebody-main}
  Suppose that $\Gamma < \mathcal{H}_g$ is a finite index subgroup. Then 
  $\Gamma$ is rigid in $\mathrm{Mcg}(\Sigma_g)$.
\end{theorem}
As a consequence we also obtain the following, which improves the main theorem of \cite{KS}.
\begin{corollary}\label{cor:comm-han}
  The abstract commensurator of $\mathcal{H}_g$ is equal to  $\mathcal{H}_g$.
\end{corollary}

The mapping class group itself, and its finite index subgroups are
rigid in all but a few exceptional low-complexity cases. These results
have a long history, starting with Ivanov's study of the automorphism group and commensurator
of the mapping class group \cite{Ivanov-Commensurator, Korkmaz-Commensurator}, whose methods
were later greatly extended (see \cite{Shackleton, IMcC,
  Bell-Margalit, Behrstock-Margalit} and the references therein).
Rigidity is also known for the group generated by powers of Dehn
twists \cite{Amarayona-Souto-powers}. In \cite{BM, Ki} it is shown
that the Johnson kernel is rigid inside the Torelli subgroup of the
mapping class group.

\bigskip 
We next study injections of $\mathcal{H}_g$ into higher genus mapping class groups.
Here, the situation is drastically different.
\begin{theorem}[Flexibility]\label{thm:flexibility-han}
  There is a finite index subgroup $\Gamma < \mathcal{H}_g$ and an injection
  $f:\Gamma \to \mathrm{Mcg}(\Sigma_h), h > g$, so that the image is
  not conjugate into $\mathcal{H}_h$. 
\end{theorem}
The example from Theorem~\ref{thm:flexibility-han} comes from a covering construction,
and we can completely characterize rigidity and flexibility for such injections.
\begin{theorem}[Covers]\label{thm:characterize-covers}
  Suppose that $\Sigma' \to \Sigma$ is a finite normal cover of a
  surface of genus $g\geq 3$. Let $\Gamma < \mathcal{H}_g$ be a finite
  index subgroup of mapping classes which lift to $\Sigma'$. Denote by
  $\Gamma'$ a finite index subgroup of the lifts of elements in
  $\Gamma$.

  Then $\Gamma'$ is conjugate into a handlebody group of $\Sigma'$ if and only
  if $\Sigma' \to \Sigma$ can be extended to a cover of handlebodies.
\end{theorem}
The genus restriction in this theorem is likely not required, and an artifact of our proof.

In the course of the proof of
Theorem~\ref{thm:rigidity-handlebody-main} we show rigidity for a
different group. The \emph{twist group} $\mathcal{T}_g<\mathcal{H}_g$
is the subgroup generated by Dehn twists about meridians of a
handlebody. It is known to be of infinite index, not finitely
generatable, and with infinite rank first homology \cite{McC}.
Nevertheless, rigidity holds:
\begin{theorem}
  Suppose that $\Gamma < \mathcal{T}_g$ is a finite index subgroup. Then 
  $\Gamma$ is rigid in $\mathrm{Mcg}(\Sigma_g)$. The commensurator of $\Tg$
  is the handlebody group $\Hg$.
\end{theorem}
The flexibility exhibited in Theorem~\ref{thm:flexibility-han}, and restrictions for covering
constructions as in Theorem~\ref{thm:characterize-covers} is also already true for
$\mathcal{T}_g$.

\subsection*{Methods of Proof}
The argument which is used to show rigidity results on injections $f$
between subgroups of mapping class groups goes back to Ivanov.
It has by now become somewhat standard, and consists of three main steps. 
First, one shows that powers of Dehn twists
map under $f$ to (roots of) multitwists. In this way one obtains a
map between curve graphs (or related objects). Then, one uses rigidity
results for maps between curve graphs to find a candidate conjugation
map which $f$ will be equal to. Checking this equality in a third step is then
usually straightforward.

For finite index subgroups of $\mathrm{Mcg}(\Sigma)$, the first step is well-known and due to
Ivanov (see \cite{Bell-Margalit, Irmak-I} for well-written
modern treatments of this argument). A key ingredient in his proof is that one can characterize (powers of)
Dehn twists in the mapping class group via ranks of maximal Abelian subgroups of
centralizers and centralizers of centralizers. 

In Section~\ref{sec:abundant}, we develop a variant of Ivanov's
argument, which may be of interest in studying the rigidity of other
subgroups of the mapping class group. It bypasses an explicit
identification of Dehn twists via their centralizers and also tries to
avoid using maximal Abelian subgroups as much as possible. A
reader experienced with arguments of this type who is only interested in the
handlebody group may skip directly to Section~\ref{sec:disk-rigidity}.

We want to emphasize that there is an alternative approach to this first step due to 
Aramayona--Souto \cite{AS} which would work (with minor modifications) also for the 
handlebody group (but, to the knowledge of the author, not the twist group, 
since it has infinite rank first rational cohomology
\cite{McC}).

\smallskip 
The second and third steps of the proof require new arguments in the case of
the handlebody group. In Section~\ref{sec:disk-rigidity}, we show that the disk graph of
a handlebody is rigid inside the curve graph of the surface (compare also
\cite{Amarayona-Souto-powers} for rigidity of subgraphs of the curve graph). 
This is then used to 
find the candidate conjugation, relying on the main result of \cite{KS}.

\smallskip
In Section~\ref{sec:flexible}, we prove the Flexibility and Covering theorems~\ref{thm:flexibility-han} and~\ref{thm:characterize-covers}. The proofs rely on two main ingredients: on the 
one hand, a theorem of Oertel \cite{O} characterizes which multitwists on the boundary of 
a handlebody extend to homeomorphisms of that handlebody. This allows to translate
the condition of lifts being conjugate into $\Hg$ into a condition on lifts of meridians.
Careful analysis of how intersection patterns between meridians behave under lifting
is then used to show the results.

\subsection*{Acknowledgments} The author would like Juan Souto and Dan Margalit for
enlightening discussions on rigidity of subgroups of mapping class groups. Furthermore, we would like to thank Harry Baik for
interest (and patience) during numerous discussions.

\section{Preliminaries}
\label{sec:prereqs}
In this section we collect some well-known facts that we will use
throughout.  A few conventions: all \emph{curves} will be simple,
closed and essential. When not explicitly stated otherwise, we will
identify curves with their isotopy classes. By \emph{disjointness} of
two curves we always mean disjointness up to
homotopy. \emph{Multicurves} are collections of disjoint curves, no two
of which are freely homotopic.

\subsection{Canonical Reduction Systems and Centralizers}
\label{sec:reduction}
Let $\Sigma$ be a surface of finite type, possibly with boundary
and/or marked points. We let $\mathrm{Mcg}(\Sigma)$ denote the mapping class
group, i.e. the group of homeomorphisms of $\Sigma$ up to isotopy.  
Given a mapping class $\phi \in
\mathrm{Mcg}(\Sigma)$ we say that $\phi$ is \emph{reducible} if there
is some multicurve on $\Sigma$ which is (set-wise) preserved by
$\phi$. The mapping class $\phi$ is \emph{pure} if there is a
multicurve $C$ so that $\phi$ preserves every component of $C$, and
induces on each component of $S-C$ either the identity, or a
pseudo-Anosov map.  If $\phi$ is pure, then the \emph{canonical
  reduction system} $C(\phi)$ is the (unique) smallest such
multicurve. If $\phi$ is pseudo-Anosov, we set $C(\phi) = \emptyset$.

The following is due to Ivanov (compare e.g. \cite[Theorem~1.2]{I}).
\begin{proposition}
  There is a finite index subgroup $\Gamma_p < \mathrm{Mcg}(\Sigma)$
  so that every reducible element in $\Gamma_p$ is pure.
\end{proposition}
Hence, we may define the canonical reduction system for any element $\phi$ to be
the canonical reduction system of a suitably big, pure power of $\phi$.

We need a version for subgroups as well. If $\Gamma < \mathrm{Mcg}(\Sigma)$ is
a \emph{pure subgroup}, i.e every reducible element in $\Gamma$ is pure, then we define
the canonical reduction system 
\[ C(\Gamma) = \bigcap_{\phi\in\Gamma} C(\phi) \] to be the
intersection of all canonical reduction systems $C(\phi)$ of every
element $\phi\in\Gamma$. If $\Gamma$ is not pure, we define $C(\Gamma)$ as the
canonical reduction system of a finite index pure subgroup.

$C(\Gamma)$ has the property that for each complementary component $Y$
of $C(\Gamma)$, either every pure element $\phi \in \Gamma$ restricts to
the identity in $Y$, or there is an element in $\Gamma$ which
restricts to a pseudo-Anosov map in $Y$.

\smallskip
We also use the following standard results on (non-)commuting elements in the mapping class groups.
\begin{proposition}[\cite{McCpreprint}]\label{prop:pa-centralizer}
  Let $\psi$ be a pseudo-Anosov. Then the cyclic group generated by $\psi$ is
  finite index in the centralizer of $\psi$.
\end{proposition}
In particular, if $\psi$ is a pseudo-Anosov, then no Dehn twist commutes with $\psi$, and
neither does an \emph{independent} pseudo-Anosov (i.e. one which does not admit a common
root, or alternatively, has different stable and unstable foliations).

The following facts on Dehn twists can e.g. be found in \cite[Section~3.3]{FM}.
\begin{lemma}\label{lem:dehn-commute}
  Some powers of two Dehn twists $T_\alpha$ and $T_\beta$ commute if and only if
  $\alpha$ and $\beta$ are disjoint.
\end{lemma}

\begin{lemma}\label{lem:dehn-equal}
  Two powers $T_\alpha^n$ and $T^m_\beta$ of Dehn twists are equal if
  and only if $n=m, \alpha = \beta$.
\end{lemma}

\subsection{Handlebody groups}
\label{sec:handlebody-prereq}
Let $V$ be a handlebody of genus $g$. Identify the boundary $\partial V$ of $V$ 
with a surface $\Sigma$ of genus $g$. A \emph{meridian} for $V$ is a curve $\alpha$ on $\Sigma$
which bounds a disk in $V$.
We will often use the following restriction on the intersections between meridians (see e.g. \cite{HH,Hempel,Masur} and references therein).
\begin{lemma}\label{lem:wave}
  Suppose that $A, B$ are two multicurves consisting of meridians (for some handlebody). Then,
  $A$ has a \emph{wave} with respect to $B$: there is a sub-arc $a\subset \alpha \in A$ which intersects
  $B$ exactly in its endpoints, and at both ends approaches the component of $B$ which it
  intersects from the same side.

  Furthermore, there is a sub-arc $b \subset \beta \in B$ so that $a \cup b$
  is a meridian.
\end{lemma}

The restriction map induces a homomorphism
\[ \mathrm{Mcg}(V) \to \mathrm{Mcg}(\Sigma) \]
whose image $\mathcal{H}_g$ we call the \emph{handlebody group} of $V$.
Up to conjugation, $\mathcal{H}_g$ is independent of the identification of $\partial V$
with $\Sigma$. Usually, we will not need to distinguish between different conjugates,
and fix some handlebody group  $\mathcal{H}_g$. In any case, the statement
that some group is conjugate into $\mathcal{H}_g$ is well-defined without choices.

\smallskip
A \emph{reduced disk system for $V$} is a multicurve $\alpha_1, \ldots, \alpha_g$ consisting of
meridians so that $\Sigma-(\alpha_1\cup\dots\cup\alpha_g)$ is connected. Note that 
every simple closed curve which is disjoint from a reduced disk system is a meridian. This is due to the fact that any curve on the boundary of a ball bounds a disk in the ball.
The following is
standard, and an immediate consequence of the fact that any homeomorphism of a sphere
extends to the ball it bounds.
\begin{lemma}
  Suppose that $\phi \in \Mcg(\Sigma)$ is such that $\phi(C)$ is a
  reduced disk system for $V$ for some reduced disk system $C$ for
  $V$. Then $\phi \in \Hg$.
\end{lemma}
We also require a standard method to transform one reduced disk system into another. See
e.g. \cite{HH,Hempel,Masur} (and references therein) for proofs.
\begin{lemma}\label{cut-system-graph-connected}
  Let $C, C'$ be reduced disk systems for $V$. Then there is a sequence
  \[ C = C_1, C_2, \ldots, C_n = C' \]
  of reduced disk systems for $V$ so that $C_i, C_{i+1}$ are disjoint for all $i$.
\end{lemma}

\smallskip
We need the following criterion for a multitwist to be an element of $\mathcal{H}_g$,
which relies on \cite[Theorem~1.11]{O}.
\begin{theorem}\label{thm:characterize-multitwists}
  Let $\phi = T_{\alpha_1}\cdots T_{\alpha_n}$ be a product of
  Dehn twists about disjoint curves $\alpha_i$ and suppose
  that $\phi$ is an element of $\mathcal{H}_g$. 

  Then, up to relabeling, there is a $l>0$ and a bijection $k:\{l+1,\ldots, n\}\to\{l+1,\ldots, n\}$ so that
  \begin{enumerate}[i)]
  \item $\alpha_i$ is a meridian for all $i\leq l$.
  \item $\alpha_i$ and $\alpha_{k(i)}$ are joined in $V$
    by a properly embedded annulus for all $i > l$.
  \item If $i>l$, then $T_{\alpha_i}$ and $T_{\alpha_{k(i)}}$ are not
    both left or both right Dehn twists.
  \end{enumerate}
\end{theorem}
\begin{proof}
  \cite[Theorem~1.11]{O} implies that $\phi$ is the restriction to
  $\partial V$ of a homeomorphism $F:V \to V$, which is a product of
  twists about disjoint disks and annuli in the handlebody. A twist
  about a disk in $V$ restricts to a Dehn twist about a meridian $\alpha_i$ on
  $\partial V$. 
A twist about an annulus $A$ with boundary
  $\partial A = \alpha_i \cup \alpha_{k(i)}$ restricts to the product
  of a left and a right Dehn twist about $\alpha_i$ and
  $\alpha_{k(i)}$ (or vice versa). This shows the theorem.
\end{proof}
\begin{corollary}\label{cor:no-left-multitwists}
  Suppose that $\alpha_1, \ldots, \alpha_k$ are disjoint simple closed curves. Then
  the product of left Dehn twists $T_{\alpha_1}\cdots T_{\alpha_k}$ is an element
  of $\mathcal{H}_g$ if and only if all $\alpha_i$ are meridians.
\end{corollary}

\section{Full and Abundant subsurfaces}
\label{sec:abundant}
In this section we discuss the first step of the rigidity proof outline given in the introduction.

Throughout, $\Sigma$ will be a finite type surface, possibly with boundary or cusps.
A subsurface $S \subset \Sigma$ is \emph{essential} if every component of $\partial S$
is an essential simple closed curve on $\Sigma$. If $\Gamma < \Mcg(\Sigma)$ is any
subgroup, we denote by $\Stab_\Gamma(S)$ the subgroup of $\Gamma$ consisting
of all elements which preserve $S$ (up to isotopy).

If $S$ is a surface with a specified collection of boundary components
$B$, we denote by $\hat{S}$ the surface obtained from $S$ by gluing
punctured disks to each boundary component of $S$ in $B$. We say that
$\hat{S}$ is obtained from $S$ by \emph{cusping off the
  boundaries $B$}. There is a homomorphism
\[ r_S:\Mcg(S) \to \Mcg(\hat{S}) \] To ease notation, we will often
say that $\phi \in \Gamma$ has a property \emph{when viewed as a
  mapping class of $\hat{S}$} if $r_S(\phi)$ has this property. Also
note that the kernel of $r_S$ consists of Dehn twists about the boundary
components $B$. See \cite[Section~4.2]{FM} for this, and related
background on mapping class groups.

\smallskip
If $S\subset \Sigma$ is an essential subsurface, then we will always denote
by $\hat{S}$ the surface obtained by cusping off all boundary components which are not
contained in the boundary of $\Sigma$.
\begin{definition}
  Let $S\subset \Sigma$ be an essential subsurface. A subgroup $\Gamma < \Mcg(\Sigma)$ is
  \begin{description}
  \item[full in $S$] if there are elements $\phi_1,\phi_2$ which are independent pseudo-Anosov
    elements when viewed as mapping classes of $\hat{S}$.
  \item[abundant in $S$] if additionally there is a pants
    decomposition $\{\alpha_1,\ldots,\alpha_k\}$ of $\hat{S}$ and
    $T_1, \ldots, T_k \in \Stab_\Gamma(S)$ so that $T_i$ is a power of
    Dehn twist about $\alpha_i$ (viewed as a mapping class of
    $\hat{S}$).
  \end{description}
\end{definition}
\begin{remark}\label{rem:abundent-passes-fi}
  If $\Gamma' < \Gamma$ is finite index, and $\Gamma$ is full or abundant in $S$, then so is
  $\Gamma'$.
\end{remark}

\begin{lemma}\label{lem:abundent-centers}
  Suppose that $S \subset \Sigma$ is an essential subsurface and that
  $\Gamma<\Mcg(\Sigma)$ is full in $S$. Then every element in
  the center of $\Stab_\Gamma(S)$ has a power which is a multitwist about $\partial S$.
\end{lemma}
\begin{proof}
  Consider the induced map
  \[ r_S: \Stab_\Gamma(S) \to \Mcg(\hat{S}) \] If the center of
  $\Stab_\Gamma(S)$ contains an element none of whose powers are multitwists
  about $\partial S$, then its image is an infinite order element in
  the center of $r_S(\Stab_\Gamma(S))$. However, since we assume that
  $\Gamma$ is full in $S$, the group $r_S(\Stab_\Gamma(S))$ contains two
  independent pseudo-Anosov elements. By Proposition~\ref{prop:pa-centralizer} this is
  impossible.
\end{proof}

\begin{proposition}\label{prop:abundent-kernels}
  Suppose that $S \subset \Sigma$ is an essential subsurface and that
  $\Gamma< \Mcg(\Sigma)$ is full in $S$. Suppose $f: \Gamma \to G$ is 
  a homomorphism.

  If $\ker(f) \cap \Stab_\Gamma(S)$ contains an element none of whose
  powers are multitwists about $\partial S$, then $\ker(f) \cap
  \Stab_\Gamma(S)$ contains an element which is pseudo-Anosov when
  viewed as a mapping class of $\hat{S}$.
\end{proposition}
\begin{proof}
  Since $\Gamma$ is full in  $S$, there is an element $\psi \in \Stab_\Gamma(S)$ which
  is pseudo-Anosov as a mapping class on $\hat{S}$. Furthermore, by
  assumption, there is an element $\phi \in \ker(f) \cap
  \Stab_\Gamma(S)$, none of whose powers are multitwists about $\partial
  S$.  Thus, $\phi$ defines an infinite order mapping class on
  $\hat{S}$. If $\phi$ or any power of it is pseudo-Anosov, we are already done.
  
  Otherwise, consider $\phi \psi^n \phi^{-1}$. As a mapping class of
  $\hat{S}$ this is pseudo-Anosov. In fact, it is an independent
  pseudo-Anosov to $\psi$: by Lemma~\ref{prop:pa-centralizer}, any
  infinite order element in the centralizer of a pseudo-Anosov has a
  power which is a pseudo-Anosov itself.

  Thus, for any $n>0$, the element $\phi\psi^n\phi^{-1}\psi^{-n}$ is
  contained in $\ker(f)$ since the latter is normal. Once $n$ is large
  enough, it will also be pseudo-Anosov, since large powers of
  independent pseudo-Anosovs on $\hat{S}$ generate a purely
  pseudo-Anosov group (compare \cite{Fujiwara}).
\end{proof}
The following is the core technical result of this section.
\begin{theorem}\label{thm:complexity-reduction-kernel}
  Let $S \subset \Sigma$ be an essential subsurface and suppose that
  $\Gamma < \Mcg(\Sigma)$ is full in $S$. Let $f:\Gamma \to
  \Mcg(\Sigma')$ be an injection into another mapping class group.

  Then there is a complementary component $Y$ of the canonical
  reduction system $C(f(\mathrm{Stab}_\Gamma(S)))$ and
  a finite index subgroup $\Gamma'<\mathrm{Stab}_{\Gamma}(S)$, so that the induced map
  \[ f: \Gamma' \to \Mcg(Y) \]
  is an injection. The kernel of the induced map
  $\hat{f}:\Gamma' \to \Mcg(\hat{Y})$ consists of multitwists about the boundary of $S$.
\end{theorem}
\begin{proof}
  We may assume that $f(\Gamma)$ is pure (otherwise, pass to a
  suitable finite index subgroup $\Gamma'$). Put
  $C = C(f(\mathrm{Stab}_\Gamma(S)))$.  If there is only one
  complementary component of $C$, there is nothing to prove for the
  first claim. The second claim follows since any $\phi$ so that
  $\hat{f}(\phi)=1$ is contained in the center of $\mathrm{Stab}_\Gamma(S)$.

  \smallskip
  Hence, suppose that there is more than one complementary component of $C$;
  let $\Sigma_1, \Sigma_2$ be two disjoint nonempty unions of complementary components
  of $C$ whose union is all of $\Sigma'$. Denote by $\widehat{\Sigma_i}$ the surface
  obtained from $\Sigma_i$ by cusping off the boundary components corresponding to 
  curves in $C$.

  The injection $f$ induces maps
  \[ f_i : \Stab_\Gamma(S) \to \Mcg(\Sigma_i), \quad i=1,2 \]
  and 
  \[ \hat{f}_i : \Stab_\Gamma(S) \to \Mcg(\widehat{\Sigma_i}), \quad i=1,2. \]
  By induction, it suffices to show that one of the $f_i$ is injective.

  We define the product map
  \[ \hat{f}=\hat{f}_1\times\hat{f}_2:\Stab_\Gamma(S) \to \Mcg(\widehat{\Sigma_1}) \times
  \Mcg(\widehat{\Sigma_2}). \] 
  Suppose that $\phi$ is such that $\hat{f}(\phi) = 1$. Then
  $f(\phi)$ is a multitwist about $C$, and thus commutes with every
  element in $f(\Stab_\Gamma(S))$. As $f$ is injective, this implies
  that $\phi$ commutes with every element of $\Stab_\Gamma(S)$, and
  thus $\phi$ has a power which is a multitwist about the boundary of $S$
  by Lemma~\ref{lem:abundent-centers}.

  Suppose now that $\ker(\hat{f}_i)$ contains elements which do not
  have powers which are multitwists about $\partial S$ for both
  $i=1,2$. We let $\phi_1, \phi_2$ be mapping classes so that
  \begin{enumerate}
  \item $\phi_i \in \ker(\hat{f}_i)$.
  \item $\phi_i$ is pseudo-Anosov on $S$.
  \end{enumerate}
  whose existence is guaranteed by Proposition~\ref{prop:abundent-kernels}. 

  Then, $f_2(\phi_1)$ is a multitwist about $\partial \Sigma_2$ and 
  $f_1(\phi_2)$ is a multitwist about $\partial \Sigma_1$ and therefore
  \[ 1 = [f(\phi_1), f(\phi_2)] = f([\phi_1, \phi_2]). \]
  Since $f$ is injective this implies that $\phi_1$ and $\phi_2$ commute. Thus,
  $\phi_1, \phi_2$, seen as mapping classes of $\hat{S}$, are commuting
  pseudo-Anosovs and thus powers of a common pseudo-Anosov of
  $\hat{S}$. By passing to powers we may therefore assume that $\phi_1
  = \psi T_1, \phi_2 = \psi T_2$ for some $\psi$
  pseudo-Anosov on $\hat{S}$ and $T_1, T_2$ multitwists about $\partial S$. 
  In other words, $\phi_2 = \phi_1 m$ for some $m$ in the center of $\Stab_\Gamma(S)$.

  If $\phi_1 = \phi_2$, then $\hat{f}(\phi_1)=1$ and hence $f(\phi_1)$ is
  central. By injectivity of $f$ this would imply that $\phi_1$ is
  central, which is impossible by Lemma~\ref{lem:abundent-centers}.

  Otherwise, $\phi_1^{-1}\phi_2 = m \neq 1$ is a nontrivial central element.  
  Since $\Gamma$ is full in $S$, there is an element $\rho$ so that
  \[ [\phi_1, \rho] \neq 1 \] (there are independent pseudo-Anosovs,
  hence not every element can commute with $\phi_1$).  
  By injectivity of $f$, we therefore have that
  \[ [f(\phi_1), f(\rho)] \neq 1 \]
  Since the kernel of $\hat{f}$ is central, this implies that
  \[ [\hat{f}(\phi_1), \hat{f}(\rho)] \neq 1 \]
  As $\hat{f}_1(\phi_1) = 1$ this means that 
  \[ [\hat{f}_2(\phi_1), \hat{f}_2(\rho)] \neq 1 \]
  But since $\hat{f}_2(\phi_1) = \hat{f}_2(m^{-1})$, this would imply that $\hat{f}_2(m)$
  and $\hat{f}_2(\rho)$ do not commute, contradicting the fact that $m$ is central.

  \smallskip 
  This contradiction shows that we may assume (up to
  relabeling) that $\ker(\hat{f}_1)$ contains only contains elements
  which do have a power which is a multitwist about $\partial S$.

  Suppose now that $m$ is an element of $\ker(f_1)$. By the above, it
  has a power $m^k$ which is a multitwist about $\partial S$, and therefore
  central in $\mathrm{Stab}_\Gamma(S)$. Hence, either $m$ is finite order,
  or $f(m^k)$ is nontrivial multitwist about $C(f(\mathrm{Stab}_\Gamma(S)))$. In the
  latter case $f_1(m) \neq 1$.
  Taking $\Gamma'$ so that $\mathrm{Stab}_\Gamma(S)$ is torsion-free therefore
  shows the theorem.
\end{proof}

Using this result we can show that, under suitable assumptions, images of 
Dehn twists are roots of Dehn twists. In the proof we require the notion
of the \emph{complexity}
$\xi(F)$ of a finite type surface. Namely, $\xi(F) = 3g(F) + 2b(F) - 3$ is the
number of curves in a pants decomposition for $F$, where $g(F)$ is the genus
and $b(F)$ is the number of boundary components and cusps.
\begin{corollary}\label{cor:nonsep-twists-to-nonsep-twists}
  Let $\gamma$ be a non-separating simple closed curve on $\Sigma$, and let $S$ be 
  the complement of $\gamma$. 
  Suppose that $\Gamma$ is abundant in $S$, and that $f:\Gamma \to \Mcg(\Sigma)$
  is any injection. Then $f(T_\gamma)$ has a power which is a Dehn twist about some
  non-separating curve $\delta$.
\end{corollary}
\begin{proof}
  Using that $\Gamma$ is full in $S$, we can 
  replace $\Gamma$ with the finite index subgroup $\Gamma'$ from
  Theorem~\ref{thm:complexity-reduction-kernel} and there is then
  a complementary component $Y \subset \Sigma$ and an injective map
  \[ \mathrm{Stab}_\Gamma(S) \to \Mcg(Y) \]
  so that the induced map
  \[ \mathrm{Stab}_\Gamma(S) \to \Mcg(\hat{Y}) \]
  has a kernel consisting only of twists about $\gamma$.  Using that
  $\Gamma$ is abundant in $S$, there is a free Abelian group of rank
  $\xi(\Sigma)-1$ in $\mathrm{Stab}_\Gamma(S)$ which does not contain
  any twist about $\gamma$, and therefore there is such an Abelian
  group in in $\Mcg(\hat{Y})$. Thus, $\hat{Y}$ has at least complexity
  $\xi(\Sigma)-1$, but is obtained by taking a subsurface of $\Sigma$
  and cusping off boundaries. This is only possible if $Y$ is the
  complement of a single non-separating curve $\delta$, and thus the
  canonical reduction system of $f(\mathrm{Stab}_\Gamma(S))$ is a
  single curve. This implies that some power of $T_\gamma$ -- which
  maps to a central element in $f(\mathrm{Stab}_\Gamma(S))$ -- is a
  Dehn twist about $\delta$.
\end{proof}

%
In fact, by induction, we also obtain the following result. For its statement, recall that
a \emph{cut system} is a multicurve $\alpha_1, \ldots, \alpha_g$ on a 
surface so that the complement $\Sigma-(\alpha_1\cup\dots\cup\alpha_g)$ is connected
and has genus $0$.
\begin{corollary}\label{cor:cut-to-cut}
  Suppose that $\alpha_1, \ldots, \alpha_g$ is a cut system for $\Sigma$. Assume
  that $\Gamma$ is abundant in $\Sigma-(\alpha_1 \cup \dots \cup \alpha_k)$ for
  all $1 \leq k \leq g$.

  Let $f:\Gamma \to \Mcg(\Sigma)$ be injective for some $\Gamma < \Mcg(\Sigma)$.
  Then $f(T_{\alpha_i})$ have powers which are Dehn twists about a cut system in $\Sigma$.
\end{corollary}
\begin{proof}
  By the previous Corollary~\ref{cor:nonsep-twists-to-nonsep-twists},
  $f(T_{\alpha_1})$ has a power which is a Dehn twist about some
  non-separating $\delta_1$, and furthermore $f$ induces an injection of
  $\mathrm{Stab}_\Gamma(\Sigma-\alpha_1)$ into
  $\Mcg(\Sigma-\delta_1)$. By the assumption on abundance, we can continue the argument
  inductively.
\end{proof}

\section{Rigidity of the disk graph}
\label{sec:disk-rigidity}
Recall that $V$ is a handlebody of genus $g$, and we have identified
the boundary $\partial V$ of $V$ with a surface $\Sigma$ of genus $g$.

The \emph{curve graph of $\Sigma$} is the simplicial complex
$\mathcal{C}(\Sigma)$ whose $k$--simplices correspond to multicurves
with $k+1$ components.  The \emph{disk graph $\mathcal{D}(V)$ of $V$}
is the full sub-complex of the curve graph $\mathcal{C}(\Sigma)$
spanned by the meridians for $V$.  Explicitly, $k$--simplices of
$\mathcal{D}(V)$ correspond to multicurves with $k+1$ components, each
of which is a meridian. We will usually identify a multicurve with the simplex
of $\mathcal{C}(\Sigma)$ or $\mathcal{D}(\Sigma)$ that it defines.
The following is obvious from the definitions.
\begin{lemma}\label{lem:links}
  Let $\alpha_1,\ldots,\alpha_k$ be a multicurve on $\Sigma$, and let
  $Y_1,\ldots,Y_l$ be its complementary components. Then the link of
  the simplex $\Delta$ defined by $\alpha_1,\ldots,\alpha_k$
  \[ \mathrm{lk}(\Delta) = \mathcal{C}(Y_1) * \dots * \mathcal{C}(Y_k) \]
  is the join of the curve graphs of the $Y_i$.
\end{lemma}

\smallskip 
In this section we show a combinatorial rigidity for the disk graph inside the curve
graph (compare also
\cite{Amarayona-Souto-powers} for a stronger definition of rigid subgraph). Recall from \cite{Irmak-I} that a \emph{superinjective} map
between (subgraphs) of curve graphs is a simplicial map $\iota$ with
the property that $\iota(\alpha)$ and $\iota(\beta)$ are joined by an
edge if and only if $\alpha$ and $\beta$ are joined by an edge.
\begin{theorem}\label{thm:diskgraph-rigidity}
  Let $\iota: \mathcal{D}(V) \to \mathcal{C}(\Sigma)$ be a superinjective simplicial map.
  Then $\iota$ is induced by a mapping class of $\Sigma$: there is a mapping class
  $\phi \in \Mcg(\Sigma)$ so that $\iota(\alpha) = \phi(\alpha)$ for all simple
  closed curves $\alpha$.
\end{theorem}
We expect that the result is also true for injective simplicial maps $\iota$, but have not
explored this (since it is not used in the sequel).
\begin{proof}
The proof has various stages. In each stage, we might modify $\iota$ by a mapping
class to ensure additional properties.
\begin{description}
\item[Reduced disk systems map to cut systems] 
  Fix a reduced disk system $\alpha_1, \ldots, \alpha_g$ for $V$. This
  defines a $(g-1)$--dimensional simplex $\Delta$ of $\mathcal{D}(V)\subset
  \mathcal{C}(\Sigma)$ whose link in $\mathcal{C}(\Sigma)$ is completely contained in
  $\mathcal{D}(V)$, and is isomorphic to the curve graph $\mathcal{C}(\Sigma_{0,2g})$
  of a $2g$--holed sphere.

  Consider the image $\iota(\alpha_1), \ldots, \iota(\alpha_g)$. This
  is a $(g-1)$--dimensional simplex $\iota(\Delta)$ in
  $\mathcal{C}(\Sigma)$. We claim that $\iota(\alpha_1), \ldots,
  \iota(\alpha_g)$ is non-separating, hence a cut system. Namely,
  suppose that the complement had components $Y_1, \ldots,
  Y_k$. Choose some curve $\delta$ disjoint from $\alpha_1, \ldots,
  \alpha_g$, and assume that $\iota(\delta)$ is a curve contained in
  $Y_1$. Then, any $\delta'$ with $\delta\cap\delta' \neq \emptyset$
  satisfies $\iota(\delta) \subset Y_1$ as well (as otherwise,
  $\iota(\delta),\iota(\delta')$ would be disjoint, violating
  superinjectivity). This shows that the sub-complex of
  $\mathrm{lk}(\Delta)$ spanned by every vertex not contained in the
  star $\mathrm{st}(\delta)$ is mapped into $\mathcal{C}(Y_1)$ (under
  the identification given by Lemma~\ref{lem:links}).

  Next, choose some $\delta'$ to be distance at
  least $3$ from $\delta$ in $\mathcal{C}(\Sigma_{0,2g})$, and
  repeat the argument with $\delta'$ in place of $\delta$, to see that
  $\mathrm{st}(\delta)\cap \mathrm{lk}(\Delta)$ is also mapped into $\mathcal{C}(Y_1)$.

  Therefore, $\iota$ induces a superinjective simplicial map
  $\mathcal{C}(\Sigma_{0,2g})\cong \mathrm{lk}(\Delta) \to \mathcal{C}(Y_1)$. Since the
  dimension of the curve graph is one less than the complexity of the surface, and
  $Y_1 \subset \Sigma$, this is only possible if $Y_1$ is the only complementary
  component of $\iota(\alpha_1), \ldots, \iota(\alpha_g)$.

  Since the mapping class group of $\Sigma$ acts transitively on the
  set of cut systems, up to modifying $\iota$ by a mapping
  class, we may assume that $\iota(\alpha_1)=\alpha_1, \ldots, \iota(\alpha_g)=\alpha_g$
  and therefore $\iota(\alpha_1), \ldots, \iota(\alpha_g)$ is a reduced disk system for $V$.
\item[Reduced disk systems map to reduced disk systems] 
  Let $\beta_1, \ldots, \beta_g$ be a reduced disk system for $V$, which is disjoint
  from $\alpha_1, \ldots, \alpha_g$. Then,  $\iota(\beta_1), \ldots, \iota(\beta_g)$ 
  is a cut system, which is disjoint from the reduced disk system 
  $\iota(\alpha_1), \ldots, \iota(\alpha_g)$.

  Next, note that any curve disjoint from a
  reduced disk system for $V$ is also a meridian for $V$. This is simply a consequence of
  the fact that any simple closed curve on the sphere bounds a disk in the ball.
  Hence, $\iota(\beta_1), \ldots, \iota(\beta_g)$ is a reduced disk system as well.

  By Lemma~\ref{cut-system-graph-connected}, this inductively implies that
  $\iota(\beta_1), \ldots, \iota(\beta_g)$ is in fact a reduced disk system 
  for any reduced disk system $\beta_1, \ldots, \beta_g$ for $V$.
\item[Meridians map to meridians] Note that any meridian $\delta$ is
  disjoint from some reduced disk system $\beta_1, \ldots, \beta_g$,
  and hence $\iota(\delta)$ is curve disjoint from the reduced disk
  system $\iota(\beta_1), \ldots, \iota(\beta_g)$ -- and hence a
  meridian (by the same argument as above). This implies that $\iota$
  is now a superinjective self-map of the disk graph $\mathcal{D}(V)$.
\item[Surjectivity of $\iota$] As a first step, we prove that $\iota$
  is locally surjective in the following sense.  Suppose that
  $\alpha_1, \ldots, \alpha_g$ is a reduced disk system for $V$.  By the previous steps,
  $\iota(\alpha_1), \ldots, \iota(\alpha_g)$ is also a reduced disk system
  for $V$.  Arguing as above, $\iota$ induces a superinjective map between links,
  which can be interpreted as a superinjective map
  \[\mathcal{C}(\Sigma_{0,2g})\to\mathcal{C}(\Sigma_{0,2g}).\]
  By Theorem~2 of \cite{Bell-Margalit} such a
  map is induced by a mapping class, and thus in particular
  surjective.  This implies that every curve which is disjoint from
  $\iota(\alpha_1), \ldots, \iota(\alpha_g)$ lies in the image of $\iota$. 
  By Lemma~\ref{cut-system-graph-connected} and induction, this first implies
  that every reduced disk system for $V$ is the image of a reduced disk system
  under $\iota$. Since every meridian is disjoint from some reduced disk system,
  $\iota$ is in fact surjective.
\item[Rigidity] At this point, $\iota$ is a superinjective, surjective
  self-map of the disk graph, and therefore in particular a simplicial
  automorphism.  By the main result of \cite{KS}, it is therefore
  induced by a handlebody group element, finishing the proof.
\end{description}
\end{proof}

\section{Rigidity for the Twist and Handlebody Groups}
Recall that $\Tg < \Hg$ is defined to be the subgroup of $\Hg$ generated by
Dehn twist about meridians. By Luft's theorem \cite{luft}, $\Tg$ agrees with the
kernel of the canonical map $\Hg \to \mathrm{Out}(\pi_1(V))$ induced by
the action of homeomorphisms of $V$ on the fundamental group of $V$.

We begin with some generalities on the handlebody and twist groups.
\begin{lemma}\label{lem:fi-implies-fullish}
  Let $\Gamma < \Tg$ be finite index. Suppose that $S$ is a subsurface of $\Sigma$
  whose boundary consists of meridians, and so that $\xi(S) > 0$. Then $\Gamma$
  is abundant on $S$.
\end{lemma}
\begin{proof}
  We begin by noting that a Dehn twist about a meridian is an
  element of $\Hg$.  If $\xi(S) >0$, and $S$ is bounded by meridians,
  then there are two meridians $\alpha_1, \alpha_2$ in $S$ which fill
  $S$. Since the product $T_{\alpha_1}T_{\alpha_2}\in\Tg$ of Dehn
  twists about such curves is pseudo-Anosov, and supported in $S$,
  there is an element $\phi$ in $\mathrm{Stab}_\Gamma(S)$ which
  defines a pseudo-Anosov in $\hat{S}$. Since some power of
  $T_{\alpha_1}$ also lies in $\mathrm{Stab}_\Gamma(S)$, and does not commute with
  $\phi$, by Proposition~\ref{prop:pa-centralizer}, $\Gamma$ is full for $S$.

  Also, if $S$ is bounded by meridians, there is a pants decomposition of $S$
  consisting of meridians. This shows abundance.
\end{proof}

\begin{lemma}\label{lem:meridian-action}
  Suppose that $\phi \in \Hg$ is such that $\phi(\delta) = \delta$ for every
  meridian $\delta$. Then $\phi = \id$.
\end{lemma}
\begin{proof}
  Consider a reduced disk system $\delta_1, \ldots, \delta_g$ for the
  handlebody $V$. Denote by $S$ the subsurface obtained as the
  complement of the $\delta_i$.  By assumption, $\phi$ preserves all
  $\delta_i$, and thus $S$. Every simple closed curve in $S$ is a
  meridian, and thus $\phi$ induces the identity seen as a mapping class of $\hat{S}$. This
  implies that $\phi$ is a multitwist about the $\delta_i$. Since for
  each $i$ there is a meridian crossing $\delta_i$, it is in fact the
  trivial multitwist.
\end{proof}

Now fix a finite index subgroup $\Gamma < \Tg$ or $\Gamma < \Hg$. In
fact, the only property of $\Gamma$ we use is that for any twist
$T_\alpha$ about some meridian, a suitable power $T_\alpha^n$ is
contained in $\Gamma$.  Furthermore, assume that
$f:\Gamma \to \mathrm{Mod}(\Sigma_g)$ is a given injection. We now
follow the strategy outlined in the introduction to show that $f$ is
in fact a suitable conjugation.

In this setting, Corollary~\ref{cor:nonsep-twists-to-nonsep-twists}
implies that $f(T^n_\alpha)$ is the power of a non-separating twist for
any non-separating meridian $\alpha$ and $n$ big enough.  First, we
note that this conclusion also holds for separating meridians.
\begin{lemma}
  In the setting as above, if $\delta$ is any meridian, there is some $n>0$ so that
  $f(T^n_\delta)$ is the power of some Dehn twist.
\end{lemma}
\begin{proof}
  Let $\delta$ be arbitrary. Choose a reduced disk system $\alpha_1,\ldots,\alpha_g$ disjoint
  from $\delta$. Then, by Corollary~\ref{cor:cut-to-cut}, the twists about $\alpha_i$ map
  to twists about a non-separating multicurve $\beta_1,\ldots, \beta_g$. Thus, $f$ induces
  an injective homomorphism
  \[\hat{f}:\Gamma' \to \mathrm{Mod}(S_{0,2g}) \]
  whose domain $\Gamma'$ has the property that some power of any Dehn twist is contained
  in $\Gamma'$. By Corollary~2 of \cite{Amarayona-Souto-powers} such a map is induced by
  a surface diffeomorphism, and in particular maps Dehn twists to Dehn twists.
\end{proof}

\begin{theorem}\label{thm:rigidity-main}
  Suppose that $\Gamma < \Tg$ or $\Gamma < \Hg$ is any finite index
  subgroup, and let $f:\Gamma \to \Mcg(\Sigma_g)$ be injective. Then $f$ is
  the restriction of a conjugation by an element in the mapping class
  group $\mathrm{Mod}(\Sigma)$.
\end{theorem}
\begin{proof}
  By the lemma above, for any meridian $\delta$ and $n>0$ big enough,
  $f(T_\delta^n)$ is the power of a Dehn twist about some curve
  $\iota(\delta)$. Furthermore, if $\delta, \delta'$ are disjoint,
  then $T_\delta^n, T_{\delta'}^n$ commute, thus so do the twist
  powers about $\iota(\delta), \iota(\delta')$ -- hence they are
  disjoint (Lemma~\ref{lem:dehn-commute}). In other words, $\iota$
  defines a simplicial map
\[ \iota:\mathcal{D}(V) \to \mathcal{C}(\Sigma).\]
Since $f$ is injective, this map $\iota$ is superinjective: if $\delta,\delta'$ are not disjoint, then
$T_\delta^n, T_{\delta'}^n$ and hence $f(T_\delta^n), f(T_{\delta'}^n)$ do not commute; hence
$\iota(\delta), \iota(\delta')$ are not disjoint.

By Theorem~\ref{thm:diskgraph-rigidity}, $\iota$ is therefore induced by some mapping
class of $\Sigma$. Changing $f$ by a conjugation we may 
  therefore assume that $f(T_\delta^{n(\delta)}) = T_\delta^{m(\delta)}$ for every $\delta$.

  Now, let $\phi \in \Gamma$ be arbitrary. Note that for any meridian $\delta$
  \[ T_{\phi(\delta)}^{m(\phi\delta)} =
  f(T_{\phi(\delta)}^{n(\phi\delta)}) = f(\phi
  T_\delta^{n(\phi\delta)} \phi^{-1}) = f(\phi) f(T_\delta^{n(\phi\delta)}) f(\phi)^{-1}\]
  and therefore
  \[ T_{\phi(\delta)}^{m(\phi\delta)n(\delta)} =
  f(\phi) T_\delta^{m(\delta)n(\phi\delta)} f(\phi)^{-1}
  = T_{f(\phi)(\delta)}^{m(\delta)n(\phi\delta)} \] and therefore
  $f(\phi)(\delta) = \phi(\delta)$ for all meridians $\delta$ (by
  Lemma~\ref{lem:dehn-equal}).  This implies $f(\phi) = \phi$ by
  Lemma~\ref{lem:meridian-action}.
\end{proof}

\begin{corollary}
  The abstract commensurator of $\Hg$ is $\Hg$.  
  The abstract commensurator of $\Tg$ is $\Hg$. 
\end{corollary}
\begin{proof}
  In light of Theorem~\ref{thm:rigidity-main} the only claim to prove is the following:
  suppose that $\Gamma <\Hg$ or $\Tg$ is finite index, and $\phi$ is a mapping
  class such that $\phi\Gamma\phi^{-1} < \Hg$, then $\phi \in \Hg$. To show this,
  let $\delta$ be any meridian, and $n>0$ be such that $T_\delta^n\in\Gamma$. Then
  by assumption $T_{\phi(\delta)}^n = \phi T_\delta^n \phi^{-1} \in \Hg$, and therefore
  $\phi(\delta)$ is a meridian. Hence, $\phi$ is a mapping class which preserves the
  set of meridians for $V$, and therefore contained in $\Hg$.
\end{proof}

\section{Flexibility for the handlebody group}
\label{sec:flexible}
The first goal of this section is to prove the following.
\begin{theorem}\label{thm:flexible}
  There is a finite index subgroup $\Gamma < \Hg$ and an inclusion
  $f:\Gamma \to \Mcg(\Sigma_h)$ whose image is not conjugate into
  $\mathcal{H}_h$.
\end{theorem}
The construction is very explicit and uses finite covers. The strategy is 
to consider a cover in which meridians lift to curves whose intersection 
pattern is incompatible with being meridians (or even boundaries of annuli). For an example in
genus $2$, consider Figure~\ref{fig:3cover}. While Theorem~\ref{thm:flexible}
formally can be concluded quickly from Theorem~\ref{thm:covers-main}, it is
instructive to consider the (simpler) setting of Theorem~\ref{thm:flexible} 
first, to understand the argument involved.
\begin{figure}
  \centering
  \includegraphics[width=0.6\textwidth]{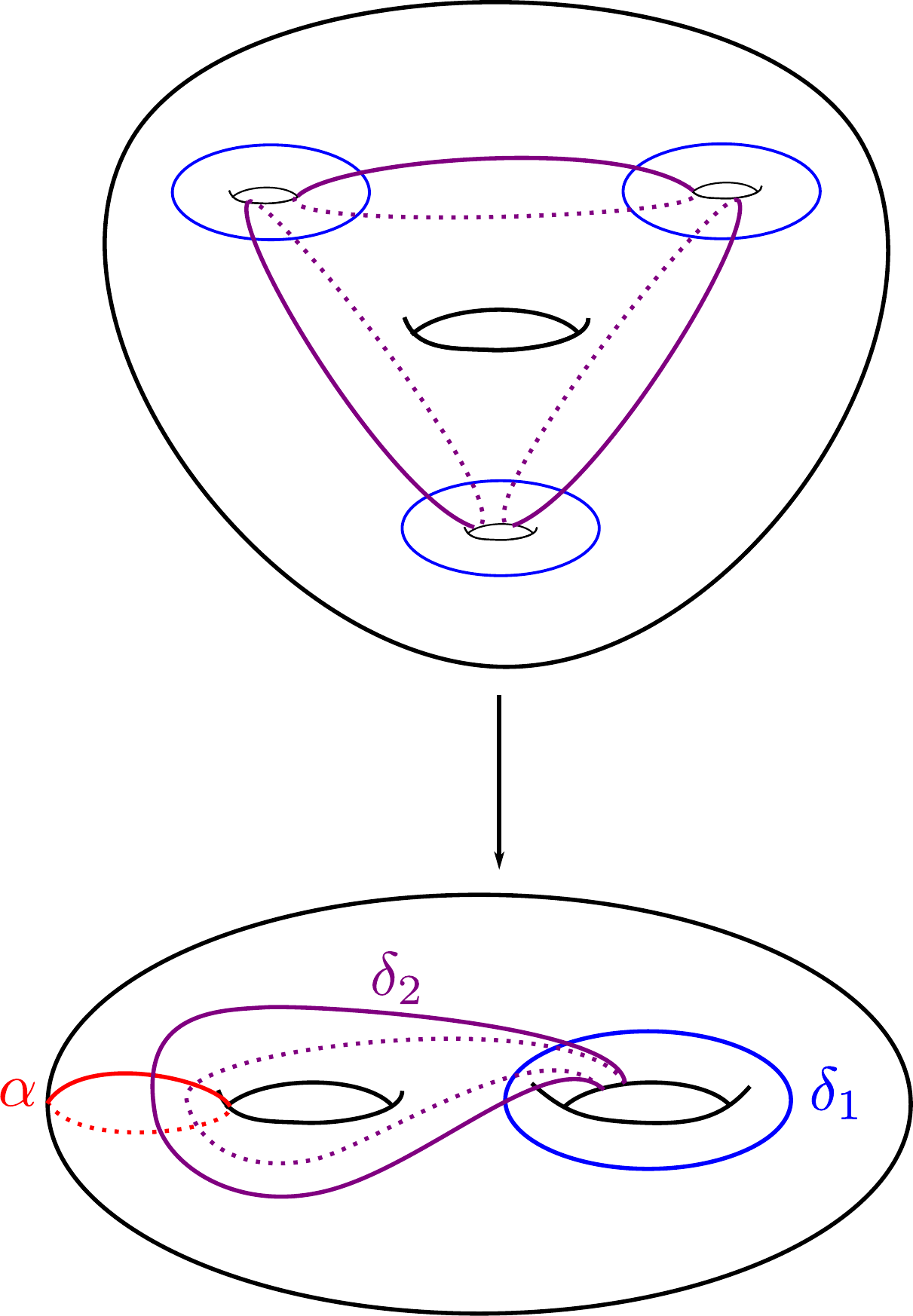}
  \caption{A 3-fold cover inducing odd intersections of meridians. The handlebody for the
    bottom surface is the ``outside'' handlebody of the standard Heegaard splitting of $S^3$: 
    curves are meridians if they can be contracted in the non-compact region of the page
    bounded by the surface.
  }
  \label{fig:3cover}
\end{figure}
\begin{proof}
  Let $\delta_0$ be a meridian, and $\alpha$ a curve intersecting
  $\delta$ once.  The map $\pi_1(\Sigma) \to \ZZ/3\ZZ$ defined by
  algebraic intersection number (mod $3$) with $\alpha$ defines a
  cover $\Sigma' \to \Sigma$ of degree $3$.

  Let $\delta_1$ be a meridian disjoint from $\delta, \alpha$, and let
  $\delta_2$ be a meridian which intersects $\delta_1$ in two points and
  $\alpha$ in two points, with algebraic intersection number $0$.

  Hence, the Dehn twists $T_{\delta_1}, T_{\delta_2}$ are in $\Hg$ and lift
  to $\Sigma'$. The lift of $T_{\delta_i}$ is the product of (left) Dehn twists
  about the three lifts $\delta_i^{(j)}, j=1,2,3$ of $\delta_i$ to $\Sigma'$.

  By construction, each $\delta_1^{(j)}$ intersects exactly two $\delta_2^{(k)}$;
  each in one point.

  Suppose that both the lifts $\widetilde{T_{\delta_1}}$ and
  $\widetilde{T_{\delta_2}}$ would be conjugate into the same
  handlebody group.  By 
  Theorem~\ref{thm:characterize-multitwists}, the multitwists
  $T_{\delta_i^{(1)}}T_{\delta_i^{(2)}}T_{\delta_i^{(3)}}$ are then
  products of twists about meridians and twists about annuli. 

  Since $3$ is odd and twist curves for annulus twists come in pairs,
  at least one of each of the curves involved is a meridian.
  On the other hand, as every $\delta_1^{(j)}$ intersects some $\delta_2^{(k)}$
  in one point, it is impossible that all three curves are meridians.

  Hence, we may assume that we have the following situation:
  \begin{itemize}
  \item $\delta_1^{(1)}$ is a meridian. 
  \item $\delta_1^{(1)}$ intersects $\delta_2^{(1)}$ $\delta_2^{(2)}$,
    and the latter two are connected by an annulus in the handlebody.
  \item $\delta_2^{(3)}$ is a meridian. 
  \end{itemize}
  However, in such a situation the product of left Dehn twists about 
  $\delta_2^{(1)}$ $\delta_2^{(2)}$ is not in the handlebody group, leading
  to a contradiction. 

  Thus, it is impossible that $\widetilde{T_{\delta_1}}$ and
  $\widetilde{T_{\delta_2}}$ are conjugate into the same handlebody subgroup
  of $\mathrm{Mcg}(\Sigma_h)$. Taking $\Gamma$ to be a subgroup which lifts to 
  $\Sigma'$ yields the desired inclusion.
\end{proof}
Denote by $\mathcal{C}^*(\Sigma_h)$ the \emph{multicurve graph of $\Sigma_h$},
i.e. the graph whose vertices correspond to (isotopy classes of) multicurves, and edges
correspond to disjointness. We warn the reader that this graph is different from other multicurve graphs considered
in the literature. Namely, we allow the number of elements in the multicurves to vary, and
more importantly, adjacency does not correspond to basic moves (e.g. exchange one 
curve). The graph $\mathcal{C}_h^*$ is however the natural object when considering
covering constructions. While there are strong restrictions for simplicial injections between
$n$--multicurve graphs (see e.g. \cite{FH}), any covering induces interesting simplicial injections
between multicurve graphs in our sense.
The construction employed in the proof of the previous result immediately implies
the following result on flexible inclusions of disk graphs.
\begin{corollary}
  There is a map $\mathcal{D}(V_g) \to \mathcal{C}^*(\Sigma_h)$ such that the images
  is not conjugate into any sub-graph where all vertices correspond to multi-meridians.
  The same remains true even if we allow that vertices map to 
  multicurves which are boundaries of annuli in the handlebody.
\end{corollary}

\medskip
For covering constructions one can analyze the situation completely. The goal is the
following theorem, whose proof will occupy the rest of this section.
\begin{theorem}\label{thm:covers-main}
  Suppose that $\Sigma' \to \Sigma$ is a finite normal cover, where
  $\Sigma$ is closed of genus $g\geq 3$. Let $\Gamma < \mathcal{H}_g$
  be a finite index subgroup of mapping classes which lift to
  $\Sigma'$. Denote by $\Gamma'$ a finite index subgroup of the lifts
  of elements in $\Gamma$.

  Then $\Gamma'$ is conjugate into a handlebody group of $\Sigma'$ if and only
  if $\Sigma' \to \Sigma$ can be extended to a cover of the handlebody $V$
  corresponding to $\Hg$.
\end{theorem}
One direction is easy: suppose $V' \to V$ is a cover of handlebodies, and 
$\partial V' = \Sigma' \to \Sigma = \partial V$ its boundary cover. If $F: V \to V$ is
a homeomorphism whose restriction $\phi$ to the boundary lifts to $\Sigma'$, then $F$ lifts to 
a homeomorphism of $V'$. Hence, any group $\Gamma'$ as in the statement is conjugate into
the handlebody group defined by $V'$.

The other direction is more involved. We begin with the following, which is a restatement of the final 
argument employed in proof of Theorem~\ref{thm:flexible}.
\begin{proposition}
  Suppose that $\Sigma' \to \Sigma$ is a finite cover. Let $\Gamma < \mathcal{H}_g$
  be a finite index subgroup of mapping classes which lift to $\Sigma'$. Denote by
  $\Gamma'$ a finite index subgroup of the lifts of elements in $\Gamma$. 

  If $\Gamma'$ is conjugate into a handlebody group of $\Sigma'$ then
  $\Sigma'$ can be identified with the boundary of a handlebody $V'$ so that
  for every meridian $\delta$, each component of the preimage of $\delta$ in
  $\Sigma'$ is a meridian for $V'$.
\end{proposition}
\begin{proof}
  Consider the left Dehn twist $T_\delta$ about any meridian, and
  consider a lift of $T^n_\delta$, where $n$ is big enough to ensure
  that $\delta^n$ lifts to a closed curve. The lift of $T^n_\delta$ is then product of left Dehn
  twists about the preimages $\delta_1, \ldots, \delta_k$ of $\delta^n$. By
  Corollary~\ref{cor:no-left-multitwists} this element is contained in
  the handlebody group of $\Sigma'$ if and only if all $\delta_i$ are meridians.
\end{proof}
To use this, we note the following standard lemma.
\begin{lemma}
  Let $\Sigma'\to\Sigma$ be a finite cover, and suppose that $\Sigma = \partial V$. Then
  $\Sigma'\to\Sigma$ extends to a cover of handlebodies (with base $V$) if and only if
  every meridian for $V$ lifts to $\Sigma'$ with degree $1$.
\end{lemma}
Thus, Theorem~\ref{thm:covers-main} will follow, once we can prove the following.
In its formulation, an \emph{elevation} of a simple closed curve $\alpha$ on $\Sigma$ (with
respect to a cover $p:\Sigma'\to \Sigma$) is any connected component of $p^{-1}(\alpha)$.
\begin{proposition}
  Suppose that  $\Sigma' \to \Sigma$ is a regular finite cover, and $\Sigma = \partial V$, $\Sigma' = \partial V'$.
  Assume that any elevation of a meridian for $V$ is a meridian for $V'$. Then every meridian of $V$ lifts with degree $1$.
\end{proposition}
For the rest of the section, we fix the cover  $\Sigma' \to \Sigma$ and assume that it
is given by a surjection
\[ \pi: \pi_1(\Sigma, p) \to G \]
to some finite group $G$. The core tool we use is the existence of waves (compare Lemma~\ref{lem:wave}).

The first part of the proof involves trying to construct a pair of
meridians whose elevations intersect in a manner incompatible with
being meridians. Namely, we have the following.
\begin{lemma}\label{lem:recur-contradiction}
  Suppose that there is a pair of meridians $\alpha,\beta$ intersecting only in the basepoint $p$
  such that $\pi(\beta)$ is not equal to a power of $\pi(\alpha)$ in $G$. Then there is a 
  meridian $\delta$ so that elevations of $\alpha$ and $\delta$ cannot be simultaneously be
  meridians (for any handlebody).
\end{lemma}
\begin{proof}
\begin{figure}
  \centering
  \includegraphics[width=0.9\textwidth]{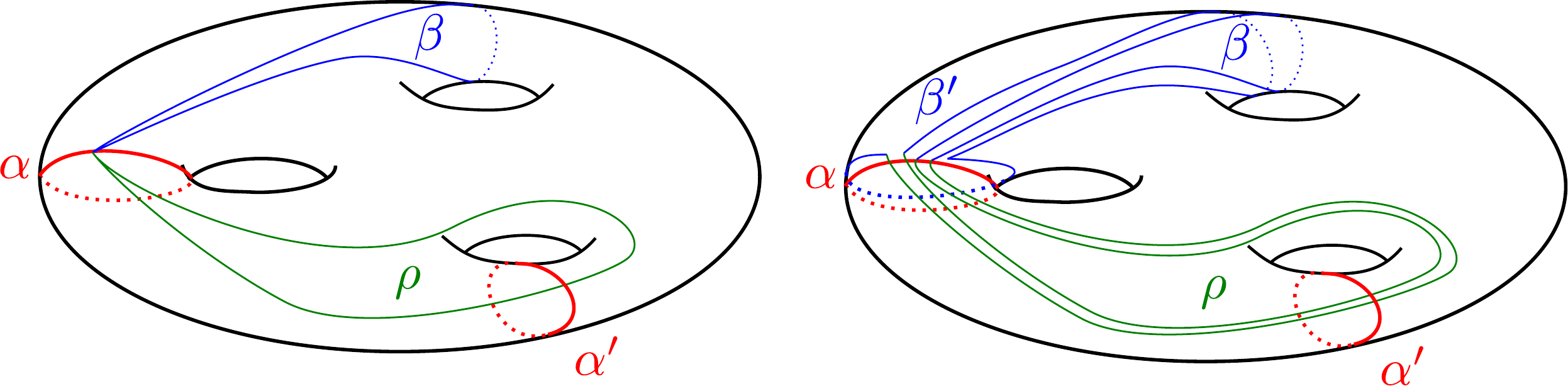}
  \caption{The construction in the proof of Lemma~\ref{lem:recur-contradiction}. The left figure shows 
    the setup; the right one the curve constructed in the proof}
  \label{fig:recur}
\end{figure}
  First note that we may assume that $\beta$ is non-separating in
  $\Sigma-\alpha$, since non-separating simple meridians generate the kernel
  $\ker(\pi_1(\partial W) \to \pi_1(W))$ for any handlebody $W$. 
  Note that $\beta'=\beta\alpha$ (or $\alpha\beta$) also has the
  property that $\pi(\beta')$ is not equal to a power of $\pi(\alpha)$
  in $G$. This means that (non-closed) lifts of $\beta$ and $\beta'$ 
  connect different elevations of $\alpha$ in $\Sigma'$.

  Next, choose a curve $\rho$ intersecting $\alpha, \beta, \beta'$ in a single point,
  and transversely intersecting a meridian $\alpha'$ disjoint from $\alpha, \beta$ in a single point (this is where we use genus $g\geq 3$ to ensure the existence of the desired $\alpha'$).
  
  The desired curve is
  \[ \delta = \beta * \rho * \beta'^{-1} * \rho^{-1} \] which is a
  simple meridian (compare Figure~\ref{fig:recur}). Consider an
  elevation $\widetilde{\delta}$ of $\delta$. By our choices,
  consecutive intersection points of $\widetilde{\delta}$ with
  components of the preimage of $\alpha\cup\alpha'$ are never 
  on the same component of $\alpha\cup\alpha'$.
  
  However, if $\widetilde{\delta}$ and all elevations of $\alpha, \alpha'$ are meridians,
  this is a contradiction, since by Lemma~\ref{lem:wave} the meridian  $\widetilde{\delta}$
  should have a wave with respect to the preimage of  $\alpha\cup\alpha'$.
\end{proof}

\begin{corollary}\label{cor:one-is-enough}
  Assume that all meridians for $V$ elevate to meridians for $V'$.
  Suppose that some meridian for $V$ lifts to $\Sigma'$ with degree $1$. Then all meridians for $V$
  lift to $\Sigma'$ with degree $1$.
\end{corollary}
\begin{proof}
  Let $\delta$ be a meridian which lifts with degree $1$. Suppose that $\delta$ is non-separating. 
  Then there are two cases: either every meridian disjoint from $\delta$ lifts with degree $1$, or not.
  In the latter case, by Lemma~\ref{lem:recur-contradiction}, there is a contradiction. In the former
  case, we argue using connectivity of the disk graph: either all non-separating meridians lift with degree
  $1$ (in which case we are done), or we eventually end up in the first case.

  Finally, suppose that $\delta$ is separating. If on either side of
  $\delta$ there is a non-separating meridian which lifts with degree
  $1$, we are done. Otherwise, argue as in the proof of
  Lemma~\ref{lem:recur-contradiction} to find a meridian $\beta$ whose elevation intersects the preimage
  of $\delta$ without waves. To do this, we simply take the concatenation of non-lifting meridians on
  either side of $\delta$.
\end{proof}

Hence, for the rest of the section, we can make the following assumption
\begin{description}
\item[AR] For any non-separating meridian $\alpha$ through $p$ , and any loop $\rho$
  which recurs to the same side of $\alpha$, and is a meridian for $V$,
  $\pi(\rho)$ is a power of $\pi(\alpha)$ in $G$.
\end{description}
Note that this first implies that any conjugate of $\alpha$ also maps to a power of $p(\alpha)$
(conjugations by all of the standard generators of $\pi_1(\Sigma)$ lie in the complement of
some meridian $\rho$ as in \textbf{AR}).
Since any two meridians can be joined by a sequence of pairwise
disjoint meridians, this implies that the image of the kernel
\[ K = \ker(\pi_1(\Sigma) \to \pi_1(V)) \]
in $G$ is cyclic, generated by some element $m$. As $K$ is normal, and $\pi_1(\Sigma)\to G$ surjective,
the subgroup $\langle m \rangle$ is therefore normal in $G$. We thus have a tower of regular coverings
\[ \Sigma' \to \Sigma'/\langle m \rangle \to \Sigma. \]
By construction, and the fact that every cyclic cover of a surface is defined by algebraic intersection with
some curve, we therefore know:
\begin{enumerate}[i)]
\item Every meridian for $V$ lifts with degree $1$ to $\Sigma'/\langle m \rangle$.
\item There is $n>0$ and a simple closed curve $\alpha \subset
  \Sigma'/\langle m \rangle$ so that a curve $\beta \subset
  \Sigma'/\langle m \rangle$ lifts to $\Sigma'$ with degree $1$ if and
  only if $i(\beta, \alpha) = 0$ mod $n$.
\end{enumerate}
\begin{lemma}
  Let $H = G/\langle m \rangle$ denote the deck group of $\Sigma'/\langle m \rangle \to \Sigma$,
  and let $\alpha$ be as in ii). Then, for any $h \in H$, we have
  \[ h[\alpha] = \pm [\alpha] \in H_1(\Sigma'/\langle m \rangle) \]
\end{lemma}
\begin{proof}
  Since the cover $\Sigma' \to \Sigma$ is normal, we have that a loop
  $\beta \subset \Sigma'/\langle m \rangle$ lifts to $\Sigma'$ with degree $1$ if and only if this
  is true for $h^{-1}\beta$, for every $h\in H$. In other words,
  \[ i(\beta, \alpha) = 0 \quad\Leftrightarrow\quad i(h\beta, \alpha) = i(\beta, h\alpha) = 0 \]
  Thus, $h\alpha$ has algebraic intersection number $0$ with exactly the same loops as $\alpha$.
  This implies that $h[\alpha]$ is a multiple of $[\alpha]$, and the multiple is $\pm 1$ as $H$ is
  finite.
\end{proof}

\begin{lemma}\label{lem:sign-swaps}
  Either some meridian for $V$ lifts to $\Sigma'$ with degree $1$, or
  the following is true: Let $\delta_1, \ldots, \delta_g$ be any
  reduced disk system for $V$, and let $Y$ be the complementary
  subsurface. Choose orientations on $\delta_i$. Let $Y' \subset
  \Sigma'/\langle m \rangle$ be a (homeomorphic) lift of $Y$, and let
  $\delta_i^+, \delta_i^-$ be lifts of $\delta_i$ on the boundary of
  $Y'$, oriented compatibly with the orientation of $\delta_i$. Then
  \[ i(\delta_i^+, \alpha) = -i(\delta_i^-, \alpha) \]
\end{lemma}
\begin{proof}
  \begin{figure}
  \centering
  \includegraphics[width=0.75\textwidth]{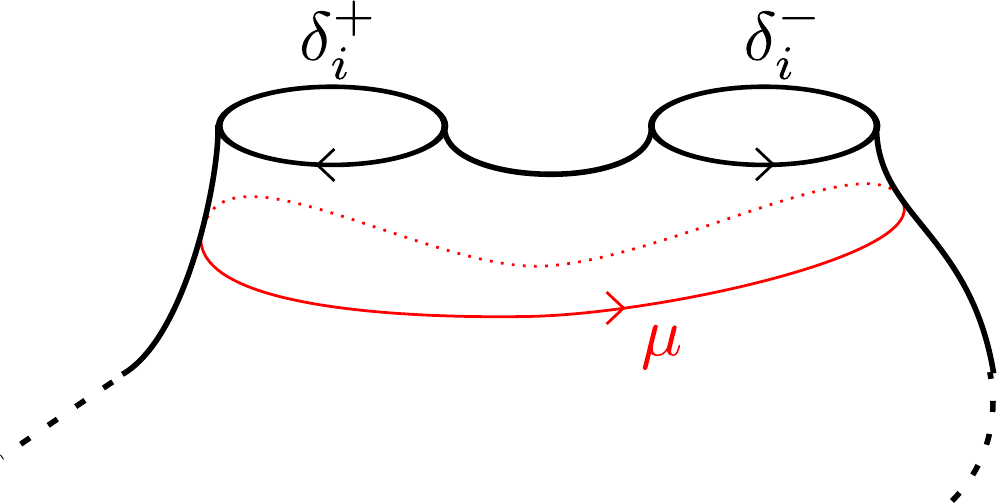}
  \caption{The construction in the proof of Lemma~\ref{lem:sign-swaps}.}
  \label{fig:sign-flips}
\end{figure}
  Since $\delta_i^+, \delta_i^-$ are (compatibly oriented) lifts of the same meridian,
  there is some element $h \in H$ so that $\delta_i^- = h\delta_i^+$. This already
  implies
  \[  i(\delta_i^+, \alpha) = \pm i(\delta_i^-, \alpha) \]
  by the above. We have to exclude the positive sign. However, note that there is a
  simple closed meridian $\mu$ in $Y'$ which is homologous to $[\delta_i^+] - [\delta_i^-]$
  (compare Figure~\ref{fig:sign-flips}).
  If in the previous equation the sign is positive, this meridian has algebraic intersection number
  $0$ with $\alpha$, therefore lifts with degree $1$ to $\Sigma'$. Being a meridian in $Y'$ it
  is also a degree $1$ lift of a simple closed meridian for $V$. This shows the lemma.
\end{proof}
\begin{lemma}\label{lem:induct-intersect}
  There is a meridian for $V$ which lifts with degree $1$ to $\Sigma'$.
\end{lemma}
\begin{proof}
  Let $\delta_1, \ldots, \delta_g$ be a reduced disk system for $V$,
  and let $Y$ be the complementary subsurface. Choose orientations on
  $\delta_i$. Let $Y' \subset \Sigma'/\langle m \rangle$ be a
  (homeomorphic) lift of $Y$, and let $\delta_i^+, \delta_i^-$ be
  lifts of $\delta_i$ on the boundary of $Y'$, oriented compatibly
  with the orientation of $\delta_i$. If any
  $i(\delta_i^\pm,\alpha)=0$, we are done. Otherwise, assume that
  $i(\delta_1^+,\alpha)>0$ and minimal among all $\delta^\pm_i$.

  Now, note that (up to possibly swapping orientation of $\delta_2$)
  there is a curve $\mu'$ in $Y'$ which is homologous to
  $[\delta_2^+] + [\delta_1^+]$. This curve $\mu'$ is a degree $1$ lift
  of a meridian $\mu$ for $V$, and we have
  \[ i(\mu, \alpha) = i(\delta_2^+, \alpha) + i(\delta_1^+,\alpha). \]
  Perform an disk system exchange move, replacing $\delta_2$ by $\mu$. Similarly,
  we modify $Y'$ by removing the pair of pants bounded by $\delta_1^+, \delta_2^+, \mu'$
  and adding a pair of pants at $\delta_2^-$, whose boundary components are
  other lifts $h_1\delta_1^+, h_2\mu'$ of $\delta_1, \mu$. Applying Lemma~\ref{lem:sign-swaps}
  twice, both for $Y'$ and its modification, we have
  \[ i(\delta_1^+, \alpha) = - i(\delta_1^-,\alpha) = i(h_1\delta_1^+, \alpha) \]
  Thus, repeating the argument, with $\mu$ in place of $\delta_2$, we can find 
  a meridian $\nu$, with lift $\nu'$ so that 
  \[ i(\nu, \alpha) = i(\delta_2^+, \alpha) + 2(\delta_1^+,\alpha). \]
  By induction, and since $\delta_1$ was chosen to minimize
  intersection with $\alpha$, after finitely many steps we will have
  found a meridian with $i(\nu, \alpha) = 0$, which is the desired
  one.
  \begin{figure}
  \centering
  \includegraphics[width=0.9\textwidth]{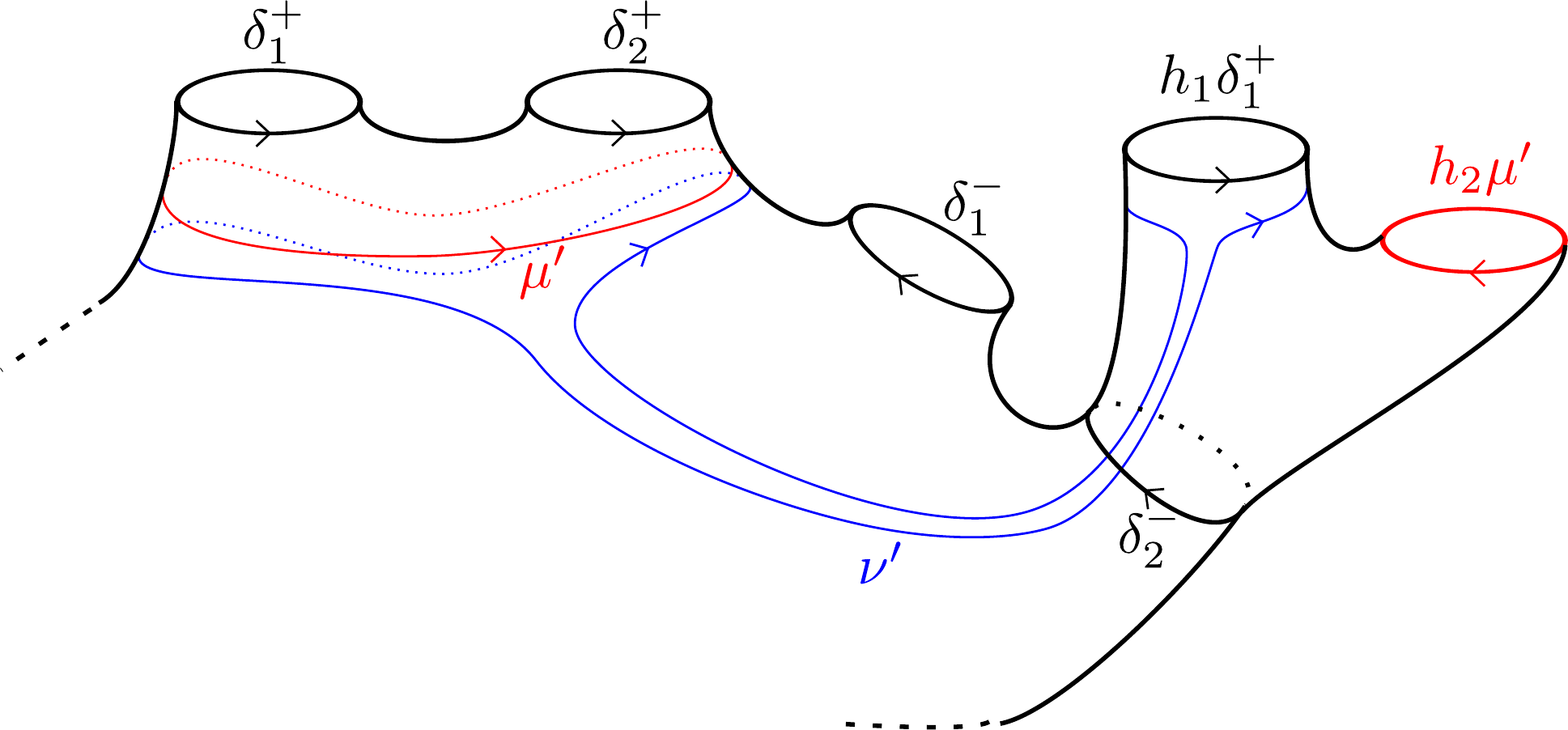}
  \caption{The construction in the proof of Lemma~\ref{lem:induct-intersect}.}
  \label{fig:induct-intersect}
\end{figure}
\end{proof}
With Corollary~\ref{cor:one-is-enough}, this is enough to finish the proof of Theorem~\ref{thm:covers-main}.

\bibliographystyle{math}
\bibliography{hbrig}

\begin{thebibliography}{McC3}

\bibitem[ALS]{ALS}
Javier Aramayona, Christopher~J. Leininger, and Juan Souto.
\newblock {{Injections of mapping class groups}}.
\newblock {\em Geom. Topol.} {\bf 13}(2009), 2523--2541.

\bibitem[AS1]{AS2}
Javier Aramayona and Juan Souto.
\newblock {{Homomorphisms between mapping class groups}}.
\newblock {\em Geom. Topol.} {\bf 16}(2012), 2285--2341.

\bibitem[AS2]{Amarayona-Souto-powers}
Javier Aramayona and Juan Souto.
\newblock {{A remark on homomorphisms from right-angled {A}rtin groups to
  mapping class groups}}.
\newblock {\em C. R. Math. Acad. Sci. Paris} {\bf 351}(2013), 713--717.

\bibitem[AS3]{AS}
Javier Aramayona and Juan Souto.
\newblock {{Rigidity phenomena in the mapping class group}}.
\newblock In Athanase {Papadopoulos}, editor, {\em {Handbook of Teichm{\"u}ller
  theory. Volume VI.}}, pages xi + 642. Z{\"u}rich: European Mathematical
  Society (EMS), 2016.

\bibitem[BM1]{Behrstock-Margalit}
Jason Behrstock and Dan Margalit.
\newblock {{Curve complexes and finite index subgroups of mapping class
  groups}}.
\newblock {\em Geom. Dedicata} {\bf 118}(2006), 71--85.

\bibitem[BM2]{Bell-Margalit}
Robert~W. Bell and Dan Margalit.
\newblock {{Injections of {A}rtin groups}}.
\newblock {\em Comment. Math. Helv.} {\bf 82}(2007), 725--751.

\bibitem[BM3]{BM}
Tara~E. Brendle and Dan Margalit.
\newblock {{Commensurations of the {J}ohnson kernel}}.
\newblock {\em Geom. Topol.} {\bf 8}(2004), 1361--1384 (electronic).

\bibitem[EF]{FH}
Viveka Erlandsson and Federica Fanoni.
\newblock {{Simplicial embeddings between multicurve graphs}}.
\newblock {\em arXiv:1510.07700} (2016).

\bibitem[FM]{FM}
Benson Farb and Dan Margalit.
\newblock {\em {A primer on mapping class groups}}, volume~49 of {\em
  {Princeton Mathematical Series}}.
\newblock Princeton University Press, Princeton, NJ, 2012.

\bibitem[Fuj]{Fujiwara}
Koji Fujiwara.
\newblock {Subgroups generated by two pseudo-{A}nosov elements in a mapping
  class group. {II}. {U}niform bound on exponents}.
\newblock {\em Trans. Amer. Math. Soc.} {\bf 367}(2015), 4377--4405.

\bibitem[HH]{HH}
Ursula Hamenst{\"a}dt and Sebastian Hensel.
\newblock {{The geometry of the handlebody groups {I}: distortion}}.
\newblock {\em J. Topol. Anal.} {\bf 4}(2012), 71--97.

\bibitem[HK]{HK}
William~J. Harvey and Mustafa Korkmaz.
\newblock {{Homomorphisms from mapping class groups}}.
\newblock {\em Bull. London Math. Soc.} {\bf 37}(2005), 275--284.

\bibitem[Hem]{Hempel}
John Hempel.
\newblock {3-manifolds as viewed from the curve complex}.
\newblock {\em Topology} {\bf 40}(2001), 631--657.

\bibitem[Irm]{Irmak-I}
Elmas Irmak.
\newblock {{Superinjective simplicial maps of complexes of curves and injective
  homomorphisms of subgroups of mapping class groups}}.
\newblock {\em Topology} {\bf 43}(2004), 513--541.

\bibitem[Iva1]{I}
N.~V. Ivanov.
\newblock {{Automorphisms of {T}eichm{\"u}ller modular groups}}.
\newblock In {\em {Topology and geometry---{R}ohlin {S}eminar}}, volume 1346 of
  {\em {Lecture Notes in Math.}}, pages 199--270. Springer, Berlin, 1988.

\bibitem[Iva2]{Ivanov-Commensurator}
Nikolai~V. Ivanov.
\newblock {{Automorphism of complexes of curves and of {T}eichm{\"u}ller
  spaces}}.
\newblock {\em Internat. Math. Res. Notices} (1997), 651--666.

\bibitem[IM]{IMcC}
Nikolai~V. Ivanov and John~D. McCarthy.
\newblock {{On injective homomorphisms between {T}eichm{\"u}ller modular
  groups. {I}}}.
\newblock {\em Invent. Math.} {\bf 135}(1999), 425--486.

\bibitem[Kid]{Ki}
Yoshikata Kida.
\newblock {{Injections of the complex of separating curves into the Torelli
  complex}}.
\newblock {\em arXiv:0911.3926} (2011).

\bibitem[Kor1]{Korkmaz-Commensurator}
Mustafa Korkmaz.
\newblock {{Automorphisms of complexes of curves on punctured spheres and on
  punctured tori}}.
\newblock {\em Topology Appl.} {\bf 95}(1999), 85--111.

\bibitem[Kor2]{K}
Mustafa Korkmaz.
\newblock {{On endomorphisms of surface mapping class groups}}.
\newblock {\em Topology} {\bf 40}(2001), 463--467.

\bibitem[KS]{KS}
Mustafa Korkmaz and Saul Schleimer.
\newblock {{Automorphisms of the disk complex}}.
\newblock {\em arXiv:0910.2038} (2009).

\bibitem[Luf]{luft}
E.~Luft.
\newblock {Actions of the homeotopy group of an orientable {$3$}-dimensional
  handlebody}.
\newblock {\em Math. Ann.} {\bf 234}(1978), 279--292.

\bibitem[Mas]{Masur}
Howard Masur.
\newblock {Measured foliations and handlebodies}.
\newblock {\em Ergodic Theory Dynam. Systems} {\bf 6}(1986), 99--116.

\bibitem[McC1]{McC2}
John~D. McCarthy.
\newblock {{Automorphisms of surface mapping class groups. {A} recent theorem
  of {N}. {I}vanov}}.
\newblock {\em Invent. Math.} {\bf 84}(1986), 49--71.

\bibitem[McC2]{McCpreprint}
John~D. McCarthy.
\newblock {{Normalizers and Centralizers of Pseudo-Anosov Mapping Classes}}.
\newblock {\em Preprint} (1994), available at
  http://users.math.msu.edu/users/mccarthy/publications/normcent.pdf.

\bibitem[McC3]{McC}
Darryl McCullough.
\newblock {{Twist groups of compact {$3$}-manifolds}}.
\newblock {\em Topology} {\bf 24}(1985), 461--474.

\bibitem[Mes]{M}
Geoffrey Mess.
\newblock {{The {T}orelli groups for genus {$2$} and {$3$} surfaces}}.
\newblock {\em Topology} {\bf 31}(1992), 775--790.

\bibitem[Oer]{O}
Ulrich Oertel.
\newblock {{Automorphisms of three-dimensional handlebodies}}.
\newblock {\em Topology} {\bf 41}(2002), 363--410.

\bibitem[Sha]{Shackleton}
Kenneth~J. Shackleton.
\newblock {{Combinatorial rigidity in curve complexes and mapping class
  groups}}.
\newblock {\em Pacific J. Math.} {\bf 230}(2007), 217--232.

\end{thebibliography}

\end{document}